\documentclass[12pt]{article}
\usepackage{amssymb,amsmath,amscd,amsthm,pb-diagram, url}
\usepackage[all]{xy}
\title{Any finite group acts freely and homologically trivially on a product of spheres}
\author{James F. Davis}
\date{}

\newcommand{\C}{{\mathbb C}}

\newcommand{\Q}{{\mathbb Q}}
\newcommand{\R}{{\mathbb R}}
\newcommand{\Z}{{\mathbb Z}}

\newcommand{\wt}[1]{{\widetilde{#1}}}
\newcommand{\step}[1]{{\noindent Step {#1}:}}

\DeclareMathOperator{\Hom}{{Hom}}

\DeclareMathOperator{\lift}{Lift}
\DeclareMathOperator{\rk}{{rk}}
\DeclareMathOperator{\frk}{{frk}}
\DeclareMathOperator{\hfrk}{{hfrk}}
\DeclareMathOperator{\UCT}{{UCT}}
\newtheorem{theorem}{Theorem}

\newtheorem{corollary}[theorem]{Corollary}
\newtheorem{lemma}[theorem]{Lemma}
\theoremstyle{definition}

\newtheorem{remark}[theorem]{Remark}

\begin{document}
\maketitle

\section{Introduction}

We state our main theorem,  give a short survey of free actions on a product of spheres,   give the title of this paper as a corollary, and finally give a proof of the theorem below.

\begin{theorem} \label{main}
Suppose $K$ is a finite CW-complex with finite fundamental group $G$.  Suppose its universal cover $\wt K$ is homotopy equivalent to a product of spheres 
$$
X = S^{n_1} \times \cdots \times S^{n_k}$$
 with all $n_i > 1$.  Then for any $n \geq  \dim X$, $G$ acts smoothly and freely on $X \times S^n$.  If $G$ acts homologically trivially on $\wt K$, then the $G$-action  on $X \times S^n$ is homologically trivial.
\end{theorem}

The study of free actions of finite groups on spheres was a motivation for early developments in algebraic $K$-theory and a playground for surgery theory.  The quintessential result is due to Madsen-Thomas-Wall \cite{MTW}: a finite group acts freely on some sphere if and only if for all primes $p$, all subgroups of order $2p$ and $p^2$ are cyclic.  

It is a natural generalization to investigate free actions of finite groups on a product of spheres.   We always assume our spheres are simply-connected, that is, a sphere is $S^n$ for $n > 1$.  Let $G$ be a finite group.  The {\em rank of $G$} is the largest integer $k$ so that there exists a prime $p$ and a subgroup of $G$ isomorphic to $(\Z/p)^k$.  The {\em free rank of $G$} is the minimal $k$ so that $G$ acts freely on a $k$-fold product of spheres.  The rank conjecture states that for  finite groups whose rank is bigger than one, the rank of $G$ equals the free rank of $G$.  A full solution to this conjecture seems elusive.

A secondary conjecture is that every finite group acts freely and homologically trivially on a product of spheres.  Here are four motivations for the conjecture and in particular for preferring homologically trivial actions.  First, homologically trivial actions arise naturally in the study of free group actions on a sphere.  The Lefschetz fixed point theorem shows that every free action on an odd-dimensional sphere is homologically trivial and that free actions on even-dimensional spheres are rather dull -- the only non-trivial group that arises is cyclic of order two.   Second,  in most
of the examples where the inequality $\text{rk }G \geq \text{frk }G$ has been proved, the group actions constructed have been homologically trivial.  For example, a faithful, finitely generated $\C[G]$-module $W$ has {\em fixity $f$} if $f = \max \{\dim W^g \mid g \in G, g \not = e\}$.  It is easy to show that $f +1 \geq \rk G$.  If $f = 0$ then $G$ acts freely on $S(W)$, while if $f$ is 1, respectively 2, it has been recently shown in \cite{ADU}, respectively \cite{OU10},  that $G$ acts freely and homologically trivially and on product of two, respectively three, spheres.  Also, the actions constructed in \cite{H} are homologically trivial.  Third, there is a homotopy theoretic reason for preferring homologically trivial actions.  If a finite group $G$ of order $q$ acts freely on a simply-connected $\Z_{(q)}$-local CW complex $Z$, then the action is homologically trivial if and only if the space $Z$ (with the original $G$-action) is equivariantly homotopy equivalent to the $G$-space $Z_t \times EG$ where $Z_t$ denotes the space $Z$ with a trivial $G$-action.  This observation is key for the technique of propagation of group actions \cite{DW}.  Finally, it is not difficult to show that any finite group $G$ acts freely on a product of $S^3$'s.  Indeed $G$ acts freely on $\coprod_{g \in G} i_! S^3$ where $S^3$ is given a free $\langle g \rangle$-action and $i_! S^3$ is the co-induced $G$-space, $i_! S^3 = \text{map}_{\langle g \rangle}(G, S^3)$.  Alas, this action is far from homologically trivial.

The question as to whether any finite group acts freely and homologically trivially was mentioned as an open question in \cite{OU10} and \cite{OU11} and was mentioned by the author as an open problem in the problem session of the 2005 BIRS conference Homotopy and Group Actions.   Motivated by this,  \" Ozg\"un \" Unl\" u and Erg\"un Yal\c c\i n proved the following theorem.

\begin{theorem}[\cite{OU12}]  \label{OY} Let $G$ be a finite group.  If $G$ has a  faithful complex representation with fixity $f$, then $G$ acts freely, cellularly, and homologically trivially on a finite complex which has the homotopy type of a $(f+1)$-fold product of  spheres.
\end{theorem}

Although we only use the statement of Theorem \ref{OY}, we discuss its proof  in order to contrast its algebraic topological techniques with the geometric topological techniques needed for our own result.   Let $W$ be a faithful, finitely generated $\C G$-module with fixity $f$.   Then $G$ acts freely and homologically trivially on the Stiefel manifold $V_{f+1}(W)$ which has both the homology and the rational homotopy type of a $(f+1)$-fold product of  spheres.   \" Unl\" u and Yal\c c\i n use this as inspiration and inductively construct finite $G$-CW-complexes $X_1, X_2, \ldots, X_{f+1}$ with $X_1 = V_1(W)$ and so that $X_i$ has the homotopy type of an $i$-fold product of spheres and whose isotropy is contained in the isotropy of $V_i(W)$.   The basic idea (but oversimplified) is that Stiefel manifolds are iterated sphere bundles, and any sphere bundle over a finite complex becomes fiber homotopically trivial after a finite Whitney sum since the stable homotopy groups of spheres are finite.

Since every finite group admits a faithful representation complex representation with finite fixity (take $V = \C[G]$), the following is a corollary of our main theorem and the theorem of   \" Unl\" u and Yal\c c\i n.  

\begin{corollary}
 A finite group acts smoothly, freely, and homologically trivially on a product of spheres.
\end{corollary}
 
Our main theorem sheds some light on the rank conjecture, albeit at the cost of an extraneous sphere.  Define the {\em homotopy free rank $\hfrk  G$} to be the minimal $k$ so that $G$ acts freely and cellularly on a finite $CW$-complex having the homotopy type of a $k$-fold product of spheres.  Progress has been made on the variant conjecture that $\text{rk } G =\text{hfrk } G$ for groups whose rank is greater than one.  In particular, Adem and Smith \cite{AS} proved this for rank 2 $p$-groups and Klaus \cite{Klaus} proved it for rank 3 $p$ groups with $p$ odd.  

As a corollary of our main theorem, one has:

\begin{corollary}
 $\frk G \leq 1+ \hfrk G$  for any finite group $G$.
\end{corollary}

Given a free action of a finite group on a $CW$-complex $Z$ having the homotopy type of a $k$-fold product of spheres, one wonders why the surgery theoretic machine does not apply to construct a free action on an honest $k$-fold product of spheres.  It is not difficult to show that the orbit space is a Poincar\'e complex, but there does not seem to be effective techniques to determine if the Spivak bundle reduces to a vector bundle.  

\section{Proof of Theorem \ref{main}}

We now embark on the proof of our main theorem.   Suppose $K$ is a finite $CW$-complex with finite fundamental group $G$ and universal cover homotopy equivalent to a product of spheres $X$.  The proof is really quite easy; one shows that the universal cover of the boundary of a regular neighborhood of $K$ is a product of spheres.   However there are some details which we divide into eight steps.  \\

\newcounter{step} \refstepcounter{step}
\step{\arabic{step}}  \label{dimK=dimX}  We may assume that the dimension of $K$ equals the dimension of $X$.\\

Note: {\em If the reader is content with the stronger hypothesis that $n \geq \max\{\dim K, \dim X \}$, then this step may be omitted.}

Step \ref{dimK=dimX} follows from the lemma below.

\begin{lemma}
Suppose $L$ is a finite connected CW-complex with finite fundamental group $G$.  Suppose the universal cover $\wt L$ is homotopy equivalent to a finite CW-complex $Y$ whose dimension is 3 or greater.  Then $L$ is homotopy equivalent to a finite CW-complex whose dimension equals that of $Y$.  
\end{lemma}

\begin{proof}
Corollary 5.1 of Wall \cite{Wall} states that for a finite connected complex $L$ and an integer $m \geq 3$, $L$ has the homotopy type of a finite  complex of dimension $m$ if and only if $H^i(L; \Z \pi_1L) = 0$ for all $ i > m$.   
Let $\pi_e : \Z G \to \Z$ be the $\Z$-module map $\pi_e(\sum n_g g) = n_e$.  For $\pi_1 L = G$ and $L$ both finite, $\pi_e$ induces a isomorphism $\Hom_{\Z G} (C_*(\wt L), \Z G) \to \Hom_{\Z}(C_*(\wt L), \Z ), \ ~ \varphi \mapsto \pi_e \circ \varphi $ and hence an isomorphism $H^*(L; \Z \pi_1L) \xrightarrow{\cong} H^*(\wt L)$.  Since $H^*(\wt L) \cong H^*(Y)$, Wall's condition is verified.
\end{proof}

\refstepcounter{step}
\step{\arabic{step}}  \label{K_simplicial}  We may assume $K$ is  a finite simplicial complex whose dimension equals the dimension of $X$.\\

Any finite CW-complex is homotopy equivalent to a finite simplicial complex of the same dimension.  This is a consequence of the simplicial approximation theorem and the homotopy extension property, see \cite[Theorem 2C.5]{Hatcher}.  We thus replace the $K$ in our theorem by the homotopy equivalent finite simplicial complex.  
\\
 
\refstepcounter{step}
\step{\arabic{step}}  \label{regular}    Let $N(K)$ be a regular neighborhood of a simplicial embedding of  $K$ in $\R^{n+ \dim X+1}$.  Then $N(K)$ is a compact PL-manifold with boundary and the embedding $K \hookrightarrow N(K)$ is a homotopy equivalence.  
\\

This is all standard PL-topology \cite[Chapter 3]{Rourke-Sanderson}, but we will briefly review.  By general position  a constant map $K \to \R^{n+\dim X + 1}$  is homotopic to  a simplicial embedding $K \hookrightarrow \R^{n + \dim X +1}$.  Take two barycentric subdivisions, and define $N(K)$ to be the  union of all closed simplices of $ (\R^{n + \dim X +1})^{\prime\prime}
$ which intersect $K$.  Then this is a regular neighborhood of $K$, a compact $PL$-manifold with boundary  which collapses onto $K$. In particular the inclusion $K \to N(K)$ is a homotopy equivalence and $N(K) - \partial N(K)$ is an open subset of $\R^{n + \dim X +1}$.\\

\refstepcounter{step}
\step{\arabic{step}}  \label{pi_1} $\pi_1 \partial N(K) \xrightarrow{\simeq} \pi_1 N(K)$ provided $\dim N(K) - \dim K > 2$.
\\

Factor the inclusion as $\partial N(K) \xrightarrow{\alpha} (N(K)-K) \xrightarrow{\beta} N(K)$.  We show that $\alpha$ is a homotopy equivalence by a direct argument and that $\beta$ induces an isomorphism on $\pi_1$ by general position.  I could not find an explicit reference for the fact that $\alpha$ is a homotopy equivalence, but a proof is easily supplied.  Any point $x \in N(K) - K$ is contained in a closed simplex  with vertices $v_1, \dots v_a, w_1, \dots w_b$ with the $v_i \in K$ and the $w_j \in N(K)-K$ and can be expressed as  $x = \sum s_i v_i + \sum t_j w_j$ with $0 \leq s_i, t_j \leq 1$, $\sum s_i + \sum t_j = 1$, and $\sum t_j > 0$. Define a deformation retract $H : (N(K)-K) \times I \to N(K)-K$ from $N(K) - K$ to $\partial N(K)$ by 
$
H(x,t) = (1-t) \sum s_iv_i + \left( 1-t + \frac{t}{\sum t_j}\right)\sum t_jw_j  .
$ 

When $\dim N(K) - \dim K > 2$, transversality  shows that
$
\beta_* : \pi_1(N(K)-K) \to \pi_1N(K)
$
is an isomorphism.  Indeed, any element of $\pi_1 N(K)$ can be represented by a map $S^1 \to N(K)$ transverse to $K$, hence whose image is disjoint from $K$, and likewise for maps $(D^2,S^1) \to (N(K), N(K)-K)$.
\\

\refstepcounter{step}
\step{\arabic{step}}  \label{smooth} Give $N(K)$ the structure of a smooth manifold with boundary so that the smooth structure on its interior is diffeomorphic with that given by considering $N(K) - \partial N(K)$ as a open subset of $\R^{n+\dim X + 1}$.
\\

Note:  {\em If the reader is content with actions that are PL instead of smooth, then this step may be omitted.}

\begin{lemma}
Let $M$ be a PL-manifold with boundary and let $\Sigma_i$ a smooth structure on $M_i = M - \partial M$ inducing the given PL-structure.  Then there is a smooth structure $\Sigma$ on $M$ so that $(M_i, \Sigma|_{M_i})$ is diffeomorphic to $(M_i, \Sigma_i)$.  

\begin{proof}
We use the Fundamental Theorem of Smoothing Theory\footnote{The original source for the proof of the Fundamental Theorem is \cite[Theorem II.4.1]{Hirsch-Mazur}, at least for manifolds without boundary.  For a discussion of how this applies to manifolds with boundary, see \cite[Essay IV, Section 2]{Kirby-Siebenmann}.  For an analogous statement of about smoothing topological manifolds, see \cite[Theorem IV.10.1]{Kirby-Siebenmann}.} which asserts that the smooth tangent bundle gives a bijection
$$
\Phi: \mathcal{S}(M) \to \lift(\tau_M)
$$
where $\mathcal{S}(M)$ is the set of isotopy classes of smooth structures on a PL-manifold $M$, where $\tau_M: M \to BPL$ is a classifying map for the PL-tangent bundle of $M$, and where $\lift(\tau_M)$ is the set of vertical homotopy classes of lifts of $\tau_M$.  A lift is a map $\tilde \tau$ making the diagram
 $$\xymatrix{
& BO\ar[d]^{\pi}\cr
M \ar@{-->}[ur]^{\tilde \tau}\ar[r]^{\tau_M} & BPL\cr
}
$$
commute and a vertical homotopy is a homotopy through lifts.  Furthermore, we assume the  map $\pi$ (induced by the forgetful map) is a fibration.  Given a smooth structure $(M, \Sigma)$ there is a lift $\tilde \tau$ which classifies the smooth tangent bundle and any two lifts are vertically homotopic.  Then one defines $\Phi[M,\Sigma] = [\tilde \tau]$.

To apply the Fundamental Theorem, consider the diagram
$$
\xymatrix{
M_i \ar@{-->}[r]^{\tilde \tau_i} \ar[d]_i  & BO \ar[d]^\pi \cr
M \ar@<-1ex>[u]_r \ar@{-->}[ur]^{\tilde \tau} \ar[r]^{\tau_M} & BPL
}.
$$
Here $i$ is the inclusion, $r$ is a homotopy inverse for $i$ (this exists since   $\partial M$ has a collar in $M$ -- see \cite{Rourke-Sanderson}), and $\tilde \tau_i$ is a classifying map for the smooth tangent bundle of $M_i$.  One may assume 
 that $\pi \circ \tilde \tau_i    = \tau_M \circ i$ by the homotopy lifting property.   Then since $\pi \circ \tilde \tau_i \circ r = \tau_M \circ i \circ r \simeq \tau_M$, the map  $\tilde \tau_i \circ r$ is homotopic to a lift $\tilde \tau$ of $\tau_M$.  Apply the Fundamental Theorem, first  to  $M$ to give a smooth structure and then  to $M_i$ to show the diffeomorphism statement.  

%^{\tilde \tau} 

\end{proof}
\end{lemma}

\refstepcounter{step}
\step{\arabic{step}}  \label{tubular_X}  Perturb the composite homotopy equivalence $X \to \wt K \to  \wt {N(K)}_i :=(\wt {N(K)} -{\partial \wt {N(K)}})$ to a smooth embedding $X \hookrightarrow \wt {N(K)}_i$.
  Let $N(X)$ be a closed tubular neighborhood of $X$ in $\wt {N(K)}_i$.  Show $\partial N(X)$ is diffeomorphic to $X \times S^n$.  
\\

There is a homotopy equivalence $X \to \wt K$ by hypothesis.  The homotopy equivalence 
$\wt K \to   \wt {N(K)}_i$ is the composite of the universal covers of the inclusion $K \hookrightarrow N(K)$ and the homotopy equivalence $r : N(K) \to N(K)_i$ from Step \ref{smooth}.  By general position (see, e.g. \cite[Chapter 2, Theorem 2.13]{Hirsch}) the composite homotopy equivalence $X \to \wt K \to  \wt {N(K)}_i$ is homotopic to a smooth embedding since $\dim \wt {N(K)}_i \geq 2 \dim X + 1$.

We next will show the normal bundle of the embedding is trivial and conclude that  $\partial N(X)$ is diffeomorphic to $X \times S^n$.  Let $\nu$ be the normal bundle $\nu(X \hookrightarrow  \wt {N(K)}_i)$. Note $ \tau_{\wt{N(K)}_i}$ is trivial since $\wt{N(K)_i}$ is the universal cover of an open subset of Euclidean space.  
 Then
$$
\nu   \oplus  \tau_X = \tau_{\wt{N(K)}_i}|_{X} \cong \varepsilon^{n + \dim X +1}
$$
where $\varepsilon^k$ represents the rank $k$ trivial bundle.
Note also that $\tau_{S^{n_i}} \oplus \varepsilon \cong \varepsilon^{n_i + 1}$ since the normal bundle of $S^{n_i} \subset \R^{n_i+1}$ is trivial, hence
$$
\tau_X \oplus \varepsilon^k = (\tau_{S^{n_1}} \oplus \varepsilon) \times \dots \times (\tau_{S^{n_k}} \oplus \varepsilon) \cong \varepsilon^{\dim X + k}
$$
Thus
$$
\nu  \oplus \varepsilon^{\dim X + k} \cong  \nu   \oplus  \tau_X  \oplus \varepsilon^k \cong \varepsilon^{n + \dim X +  k + 1}
$$
Thus, since $\rk \nu > \dim X$, $\nu$ is trivial (see e.g. \cite[Theorem 9.1.5]{Husemoller}).

The manifold $N(K)_i$ inherits a Riemannian metric since it covers an open subset of Euclidean space.  Then $\partial N(X)$ is diffeomorphic to the sphere bundle $S(\nu)$, which is in turn diffeomorphic to $X \times S^n$, since $\nu$ is trivial.  
\\

\refstepcounter{step}
\step{\arabic{step}}  \label{h-cobordism}  Apply the h-cobordism theorem to the triad
$$(\wt {N(K)} - N(X)_i ; ~\partial \wt {N(K)}, \partial N(X)),$$
and conclude that there is a free action on $X \times S^n$.
\\

Clearly $\wt {N(K)} - N(X)_i$  is a smooth compact manifold  with disjoint boundary components $\partial \wt {N(K)}$ and $\partial N(X)$.  We next show that the manifold and each of the boundary components are simply-connected.  Step \ref{pi_1} shows that $\partial \wt {N(K)}$ is simply-connected, and  $\partial N(X) \cong X \times S^n$ is simply-connected.  We claim the composite
$$
\wt {N(K)} - N(X)_i  \hookrightarrow \wt {N(K)} - X \to \wt {N(K)}
$$
is an isomorphism on fundamental groups, hence the domain is simply-connected.  The first map is a homotopy equivalence since $N(X)$ is a tubular neighborhood of $X$ and the second map induces an isomorphism on fundamental groups by general position, as in Step \ref{pi_1}.  

To show that the simply-connected compact manifold triad is an\linebreak h-cobordism, it suffices to show that the relative integral homology of the manifold relative to one of the boundary components vanishes.  But
$$
H_*(\wt {N(K)} - N(X)_i , \partial N(X)) \xrightarrow{\cong} H_*(\wt{N(K)}, N(X)) \xleftarrow{\cong}  
H_*(\wt{N(K)}, X) \cong 0
$$
where the first isomorphism is by excision, the second is by homotopy invariance, and the third is by hypothesis.

Our simply-connected smooth h-cobordism is diffeomorphic to a product (see \cite{Milnor}).  Thus, $\partial \wt {N(K)}$ is diffeomorphic to $\partial {N(X)}$ which, 
by Step \ref{tubular_X}, is diffeomorphic to $X \times S^n$.  Thus there is a diffeomorphism $f : \partial \wt {N(K)} \to X \times S^n$.  We thus have a free $G$-action on $X \times S^n$ where $g(x,y ) : = f(g(f^{-1}(x,y)))$.
 \\

\refstepcounter{step}
\step{\arabic{step}}  \label{homology} Analyze the action on homology of $\partial \wt{N(K)}$.
\\

We start by presenting a careful argument for the following:  $N(K)_i$ is an open subset of Euclidean space, hence orientable, hence so are $N(K)$ and $\wt{N(K)}$ and the $G$-action on $\wt{N(K)}$ is orientation-preserving.  

A real vector bundle over a space $B$ determines an orientation double cover $B_o \to B$.  The bundle is {\em orientable} if $B_o \to B$ has a section.  A choice of section is called an {\em orientation}.  Given a map $f : B' \to B$ covered by a map of vector bundles, the orientation double cover of $B'$ is the pullback of the orientation double cover of $B$ along $f$.  An {\em orientation of a manifold} (possibly with boundary) is an orientation of its tangent bundle.

Let $M$ be a manifold with boundary, $M_i = M - \partial M$, and $C$ be an open collar of the boundary.  An orientation on $M_i$ determines an orientation on $C_i$, hence one on $C$, since $C \simeq C_i$.  The orientations on $M_i$ and $C$ agree on their intersection $C_i$, hence determine an orientation on $M = M_i \cup C$.  A regular $G$-cover $\wt M \to M$ is covered by a map of tangent bundles, hence an orientation of $M$ determines one of $\wt M$.  Since the $G$-action on $\wt M$ covers the identity on $M$, the $G$-action preserves the orientation on $\wt M$.

Orient $\R^{\dim X + n + 1}$.  This orients $N(K)_i$ and hence, by the above, $N(K)$ and $\wt {N(K)}$.  Since $\wt {N(K)}$ is compact, there is a fundamental class $[\wt{N(K)}]$ and a Poincar\'e-Lefschetz  isomorphism.  
$$
\cap [\wt{N(K)}] : H^{\dim X + n + 1-i}(\wt {N(K)}, \partial \wt {N(K)} ) \to H_i(\wt {N(K)})
$$
for all $i$.  Since the $G$-action preserves the orientation by the above paragraph, $G$ leaves the fundamental class invariant.  

If $G$ acts on a pair of spaces $(X,Y)$, give $H_i(X,Y)$, $H^i(X,Y)$ and\linebreak $H_i(X,Y)^* := \Hom_\Z(H_i(X,Y), \Z)$ the structure of left $\Z G$-modules using, for $g \in G$, the maps $a \mapsto g_*a$,  $(\alpha \mapsto (g^{-1})^*\alpha)$, and $(\varphi \mapsto (a \mapsto \varphi((g^{-1})_* a )))$ respectively.  With these conventions and with $a \in H_N(X,Y)$ invariant under $G$, the following maps are maps of left $\Z G$-modules.
\begin{align*}
\cap [a] & : H^i(X,Y) \to H_{N-i}(X) \\ 
\UCT & : H^i(X,Y) \to H_i(X,Y)^*  \qquad (\alpha \mapsto (a \mapsto \langle \alpha, a \rangle )
\end{align*}
If the homology of $H_j(X,Y)$ is a finitely generated free abelian group for all $j$, then the $\UCT$ map is an isomorphism.  

Thus for all $i$ and for $N = \dim X + n + 1 = \dim N(K)$, we have isomorphisms of $\Z G$-modules
\begin{multline*}
H_i(\wt K) \cong H_i(\wt {N(K)}) \cong H^{N-i}(\wt {N(K)}, \partial \wt {N(K)} ) \cong H_{N-i}(\wt {N(K)}, \partial \wt {N(K)} )^*
\end{multline*}
Now if the $G$-action on $\wt K$ is homologically trivial, then we conclude that the $G$-action on $H_{N-i}(\wt {N(K)}, \partial \wt {N(K)} )^*$ is trivial, and 
since the relative homology group is free abelian, 
 the $G$-action on $H_*(\wt {N(K)}, \partial \wt {N(K)} )$ is trivial.  By looking at the exact sequence of the pair $(\wt {N(K)}, \partial \wt {N(K)} )$, one concludes that for all $i$ there is a short exact sequence of $\Z G$ modules 
$$
0 \to A_i \to H_i(\partial \wt {N(K)}) \to B_i \to 0
$$
where $A_i$ and $B_i$ have trivial $G$-actions.  Apply $ - \otimes \Q$ and use the all $\Q G$-modules are projective (Maschke's Theorem) to conclude that $H_i(\partial \wt {N(K)}) \otimes \Q \cong A_i \otimes \Q \oplus B_i \otimes \Q$ has trivial $G$-action.  Hence the submodule
$H_i(\partial \wt {N(K)})$ also has trivial $G$-action.  This completes the proof of our main theorem.  

\begin{remark}
The $G$-action on $H_*(X \times S^n)$ could be completely analyzed even when the action of $G$ on $H_*(\wt K)$ is nontrivial.

 In the statement of our main theorem, the product of spheres could be replaced by any smooth, closed, stably parallelizable manifold.   It is also true that if a finite group $G$ acts  cellularly on a finite CW-complex $Z$ with isotropy of rank at most one, and if  $Z$ has the homotopy type of a  closed, smooth, parallelizable manifold $X$, then $G$ acts freely and smoothly on $X \times S^m \times S^n$ for some $m,n > 0$.  This follows from the first sentence of this paragraph and  Theorem 1.4 of Adem and Smith \cite{AS} which implies that $G$ acts freely and cellularly on a finite complex $Y \simeq Z \times S^m$ for some $m > 0$.  
  \end{remark}

\begin{center}
{\sc Acknowledgements}
\end{center}

I would like to thank Diarmuid Crowley and Chuck Livingston for reminding me, on separate occasions, that the interior of a regular neighborhood is an open subset of Euclidean space, and hence a smooth, oriented manifold.   I thank the referee for helpful comments.

This research has been supported by the National Science Foundation grant DMS-1210991.  The research was inspired by a visit to Bo\u gazi\c ci University, where the visit was supported by the Bo\u gazi\c ci University Foundation.

\noindent James F. Davis\\
Department of Mathematics\\ Indiana University \\Bloomington,
Indiana 47405  USA \\Email: jfdavis@indiana.edu

\end{document}